\documentclass[12pt,twoside,reqno]{amsart}
\usepackage[colorlinks=false,citecolor=blue]{hyperref}
\usepackage{mathptmx, amsmath, amssymb, amsfonts, amsthm, mathptmx, enumerate, color,mathrsfs}
\usepackage{graphicx}
\usepackage{epstopdf}

\newcommand{\vertiii}[1]{{\left\vert\kern-0.25ex\left\vert\kern-0.25ex\left\vert #1 
    \right\vert\kern-0.25ex\right\vert\kern-0.25ex\right\vert}}
\newtheorem{theorem}{Theorem}[section]
\newtheorem{lemma}{Lemma}[section]
\newtheorem{remark}{Remark}[section]

\newtheorem{corollary}{Corollary}[section]

\numberwithin{equation}{section}

\begin{document}
\title{A new collection of $n-$tuple operator inequalities}
\author{Zameddin I. Ismailov, Pembe Ipek Al, Hamid Reza Moradi, and Mohammad Sababheh}

\begin{abstract}
In this paper, we present several new bounds for the norm and numerical radius of sums of Hilbert space operators. The obtained bounds form a new collection that enriches our understanding of these bounds. We compare our bounds with the existing literature using examples that demonstrate, in general, how our results are incomparable with the known bounds. Of particular interest are the treatment of the triangle inequality, the numerical radius of operator matrices, and singular value bounds for sums of operators.
\end{abstract}

\maketitle

%------------------------------------------------------------------------------------%
\pagestyle{myheadings} \markboth{\centerline{}}
{\centerline{}} \bigskip \bigskip 
%------------------------------------------------------------------------------------%

\section{Introduction}
Let $\mathbb{B}(\mathbb{H})$ be the $C^*$-algebra of all bounded linear operators on a complex Hilbert space $\mathbb{H}$. It is customary to use $\left<\cdot,\cdot\right>$ to denote the inner product on $\mathbb{H}$, and $\|\cdot\|$ for the induced norm. The zero operator on $\mathbb{H}$ will be denoted by $O$, and the adjoint of an operator $T$ is $T^*$. The  absolute value of $T$ is defined by $|T|=(T^*T)^{
\frac{1}{2}}.$ 

The operator norm and the numerical radius of an operator $T\in\mathbb{B}(\mathbb{H})$ are two scalar quantities, defined respectively by
\[\|T\|=\sup_{\|x\|=1}\|Tx\|\;{\text{and}}\;\omega(T)=\sup|\left<Tx,x\right>|.\]
As real valued functions on $\mathbb{B}(\mathbb{H})$, both $\|\cdot\|$ and $\omega(\cdot)$ define equivalent norms on $\mathbb{B}(\mathbb{H})$, as we have \cite[Theorem 1.3-1]{Gustafson_Book_1997}
\begin{equation}\label{Eq_Equiv_1}
\frac{1}{2}\|T\|\leq\omega(T)\leq\|T\|; T\in\mathbb{B}(\mathbb{H}).
\end{equation}
Researchers in this field have devoted a considerable amount of time trying to find sharper bounds than those in \eqref{Eq_Equiv_1}. For example, it was shown in \cite{Kittaneh_Studia_2005} that
\begin{equation}
\frac{1}{4}\|\;|T|^2+|T^*|^2\|\leq\omega^2(T)\leq \frac{1}{2}\|\;|T|^2+|T^*|^2\|,
\end{equation}
as one of the most beautiful double-sided bound. The significance of this bound can be found in the cited reference. The reader is referred to \cite{Bhunia_ArchMath_2021, Haj_Georgian_2021, Hoss_Slovaca_2020, Sababheh_JAAC_2023, Sababheh_CAOT_2024} for a list of recent work on this topic.

Another identity for the numerical radius was observed a long time ago in \cite{Haagerup_PAMS_1992}, as follows
\begin{equation}\label{Eq_w_theta}
\omega(T)=\sup_{\theta\in\mathbb{R}}\left\|\Re\left(e^{i\theta}T\right)\right\|,
\end{equation}
where $\Re(\cdot)$ is the real part. This is defined, for $T\in\mathbb{B}(\mathbb{H})$, by $\Re T=\frac{T+T^*}{2}.$ The imaginary part of $T$ is defined as $\Im T=\frac{T-T^*}{2i}.$ Then any operator $T\in\mathbb{B}(\mathbb{H})$ can be written via the Cartesian decomposition $T=\Re T+i\Im T$.

Implementing \eqref{Eq_w_theta}, it was shown in \cite{Hirzallah_IEOT_2011} that if $T_1,T_2\in\mathbb{B}(\mathbb{H})$, then
\begin{equation}\label{Eq_Ident_Bl}
\omega \left( \left[ \begin{matrix}
   O & {{T}_{1}}  \\
   T_{2}^{*} & O  \\
\end{matrix} \right] \right)=\sup_{\theta\in\mathbb{R}}\|T_1+e^{i\theta}T_2\|,
\end{equation}
as an identity for the numerical radius of the operator matrix $\left[ \begin{matrix}
   O & {{T}_{1}}  \\
   T_{2}^{*} & O  \\
\end{matrix} \right]\in\mathbb{B}(\mathbb{H}\oplus\mathbb{H}).$ Operator matrices have played a vital role in the advancement of numerical radius investigation, as one can see in \cite{Abu-Omar_LAA_2015, Conde_GMJ_2024, Feki_Georgian_2023, Hajmohamadi_JMI_2018, Qiao_JMI_2022, Shebrawi_LAA_2017}.

Notice that \eqref{Eq_Ident_Bl} immediately implies
\begin{equation}\label{Eq_Bound_Bl_1}
\omega \left( \left[ \begin{matrix}
   O & {{T}_{1}}  \\
   T_{2}^{*} & O  \\
\end{matrix} \right] \right)\leq\frac{\|T_1\|+\|T_2\|}{2},
\end{equation}
as one of the sharpest upper bounds for $\omega \left( \left[ \begin{matrix}
   O & {{T}_{1}}  \\
   T_{2}^{*} & O  \\
\end{matrix} \right] \right)$ available in the literature.

In this paper, we will be interested in pursuing this path by presenting several new relations for the operator norm and numerical radius of sums of operators and operator matrices. Our results will be compared with celebrated existing results, and it will be seen how our results introduce a new collection of independent bounds that are generally incomparable with the existing ones.

Another target is to show singular value bounds in the same theme. We recall that the singular values of a compact operator $T\in\mathbb{B}(\mathbb{H})$ are the eigenvalues of $|T|.$ These eigenvalues are countably decreasing, and will be enumerated as $s_1(T)\geq s_2(T)\geq\ldots.$ The class of compact operators in $\mathbb{B}(\mathbb{H})$ will be denoted by $\mathbb{K}(\mathbb{H}).$ Recent research about singular values can be found in \cite{Kittaneh_MIA_2024,Moradi_AMP_2023,Sababheh_JMI_2022,Zhao_Filomat_2022,Zou_LAA_2017} and the references therein. 

The max-min principle for the singular values is one of the most efficient tools that states the following; see \cite[Theorem 1.5]{Simon_Book_1979} or \cite[Theorem 9.1]{Gohberg_Book_2003}.
\begin{lemma}\label{lem_maxmin}
Let $T\in\mathbb{K}(\mathbb{H})$. Then  for $j=1,2,\ldots,$
\[s_j(T)=\underset{\dim M=j}{\mathop{\max }}\,\underset{\left\| x \right\|=1}{\mathop{\underset{x\in M}{\mathop{\min }}\,}}\,\left\| Tx \right\|.\]
\end{lemma}
\section{Main results}
Now we begin our discussion of the new results, starting with the following norm-numerical radius inequality for the sum of $n$ operators. Some special cases that help better appreciate this are discussed next.

\begin{theorem}\label{Thm_n_1}
Let ${{T}_{1}},{{T}_{2}},\ldots ,{{T}_{n}}\in \mathbb B\left( \mathbb H \right)$. Then 
	\[{{\left\| \sum\limits_{k=1}^{n}{{{T}_{k}}} \right\|}^{2}}\le \left\| \sum\limits_{k=1}^{n}{T_{k}^{*}{{T}_{k}}} \right\|+\omega \left( \sum\limits_{j=1}^{n}{T_{j}^{*}}\sum\limits_{k=1}^{n}{{{T}_{k}}}-\sum\limits_{k=1}^{n}{T_{k}^{*}{{T}_{k}}} \right).\]
\end{theorem}
\begin{proof}
Since $\|T\|^2=\|T^*T\|$ for any $T\in\mathbb{B}(\mathbb{H}),$ we have
\begin{align*}
{{\left\| \sum\limits_{k=1}^{n}{{{T}_{k}}} \right\|}^{2}}&=\left\|\left(\sum\limits_{k=1}^{n}{{{T}_{k}}}\right)^*\left(\sum\limits_{k=1}^{n}T_{k}\right)\right\|\\
&=\omega\left( \left(\sum\limits_{k=1}^{n}{{{T}_{k}}}\right)^*\left(\sum\limits_{k=1}^{n}T_{k}\right)     \right)\\
&=\omega\left(\sum_{k=1}^{n}T_{k}^*T_{k}+\sum_{1\leq k\not=j\leq n}T_j^*T_{k}\right)\\
&\leq \omega\left(\sum_{k=1}^{n}T_{k}^*T_{k}\right)+\omega\left(\sum_{1\leq k\not=j\leq n}T_j^*T_{k}\right)\\
&=\left\|\sum_{k=1}^{n}T_{k}^*T_{k}\right\|+\omega\left(\sum\limits_{j=1}^{n}{T_{j}^{*}}\sum\limits_{k=1}^{n}{{{T}_{k}}}-\sum\limits_{k=1}^{n}{T_{k}^{*}{{T}_{k}}}\right),
\end{align*}
where we have used the observations that $\|T\|=\omega(T)$ for self-adjoint operators to obtain the second equality, the triangle inequality to obtain the inequality, and 
\[\sum\limits_{1\le k\ne j\le n}{T_{j}^{*}{{T}_{k}}}=\sum\limits_{j=1}^{n}{T_{j}^{*}}\sum\limits_{k=1}^{n}{{{T}_{k}}}-\sum\limits_{k=1}^{n}{T_{k}^{*}{{T}_{k}}}\]
to obtain the last line. This completes the proof.
\end{proof}

\begin{remark}
If we let $n=2$ in Theorem \ref{Thm_n_1}, we obtain
	\[\begin{aligned}
   {{\left\| {{T}_{1}}+{{T}_{2}} \right\|}^{2}}&\le \left\| T_{1}^{*}{{T}_{1}}+T_{2}^{*}{{T}_{2}} \right\|+\omega \left( T_{1}^{*}{{T}_{2}}+T_{2}^{*}{{T}_{1}} \right) \\ 
 & =\left\| T_{1}^{*}{{T}_{1}}+T_{2}^{*}{{T}_{2}} \right\|+2\omega \left( \Re\left( T_{1}^{*}{{T}_{2}} \right) \right) \\ 
 & =\left\| T_{1}^{*}{{T}_{1}}+T_{2}^{*}{{T}_{2}} \right\|+2\left\| \Re\left( T_{1}^{*}{{T}_{2}} \right) \right\|.  
\end{aligned}\]
\end{remark}

\begin{remark}
If we put ${{T}_{1}}=\Re T$ and ${{T}_{2}}=i\Im T$, we obtain
	\[\begin{aligned}
   {{\left\| T \right\|}^{2}}&={{\left\| \Re T+i\Im T \right\|}^{2}} \\ 
 & \le \left\| {{\left( \Re T \right)}^{2}}+{{\left( \Im T \right)}^{2}} \right\|+2\left\| \Re\left( i\left( \Re T \right)\left( \Im T \right) \right) \right\| \\ 
 & =\left\| {{\left( \Re T \right)}^{2}}+{{\left( \Im T \right)}^{2}} \right\|+2\left\| -\Im\left( \left( \Re T \right)\left( \Im T \right) \right) \right\| \\ 
 & =\left\| {{\left( \Re T \right)}^{2}}+{{\left( \Im T \right)}^{2}} \right\|+2\left\| \Im\left( \left( \Re T \right)\left( \Im T \right) \right) \right\|.  
\end{aligned}\]
Since
	\[\left\| {{\left( \Re T \right)}^{2}}+{{\left( \Im T \right)}^{2}} \right\|=\frac{1}{2}\left\| T{{T}^{*}}+{{T}^{*}}T \right\|,\]
we get
	\[{{\left\| T \right\|}^{2}}\le \frac{1}{2}\left\| T{{T}^{*}}+{{T}^{*}}T \right\|+2\left\| \Im \left( \left( \Re T \right)\left( \Im T \right) \right) \right\|.\]
	
	Notice that this provides a reversed version of the well-known bound
	\[ \frac{1}{2}\left\| T{{T}^{*}}+{{T}^{*}}T \right\|\leq\|T\|^2.\]
\end{remark}

\begin{theorem}\label{Thm_n_2}
 Let ${{T}_{1}},{{T}_{2}},\ldots ,{{T}_{n}}\in \mathbb B\left( \mathbb H \right)$. Then 
	\[{{\left\| \sum\limits_{k=1}^{n}{{{T}_{k}}} \right\|}^{2}}\le \sum\limits_{j=1}^{n}{\omega \left( T_{j}^{*}\left( \sum\limits_{k=1}^{n}{{{T}_{k}}} \right) \right)}.\]
\end{theorem}
\begin{proof}
Following the proof of Theorem \ref{Thm_n_1}, we have
\begin{align*}
{{\left\| \sum\limits_{k=1}^{n}{{{T}_{k}}} \right\|}^{2}}&=\left\|\left(\sum\limits_{k=1}^{n}{{{T}_{k}}}\right)^*\left(\sum\limits_{k=1}^{n}T_{k}\right)\right\|\\
&=\omega\left( \left(\sum\limits_{j=1}^{n}{{{T}_{j}}}^*\right)\left(\sum\limits_{k=1}^{n}T_{k}\right)     \right)\\
&=\omega\left(\sum_{j=1}^{n}\left(T_j^*\sum_{k=1}^{n}T_{k}\right)\right)\\
&\leq \sum_{j=1}^{n}{\omega \left( T_{j}^{*}\left( \sum\limits_{k=1}^{n}{{{T}_{k}}} \right) \right)},
\end{align*}
where we have used the triangle inequality to obtain the last line. This completes the proof.
\end{proof}

\begin{remark}\label{Rem_1}
The case $n=2$ in Theorem \ref{Thm_n_2} implies
	\[{{\left\| {{T}_{1}}+{{T}_{2}} \right\|}^{2}}\le \omega \left( T_{1}^{*}\left( {{T}_{1}}+{{T}_{2}} \right) \right)+\omega \left( T_{2}^{*}\left( {{T}_{1}}+{{T}_{2}} \right) \right).\]
	If we let $T_1=\left[
\begin{array}{cc}
 3 & 3 \\
 -3 & 2 \\
\end{array}
\right]$ and $T_2=\left[
\begin{array}{cc}
 -1 & 0 \\
 -3 & -1 \\
\end{array}
\right],$ we see that 
\[{{\left\| {{T}_{1}}+{{T}_{2}} \right\|}^{2}}= \omega \left( T_{1}^{*}\left( {{T}_{1}}+{{T}_{2}} \right) \right)+\omega \left( T_{2}^{*}\left( {{T}_{1}}+{{T}_{2}} \right) \right)=40,\]
showing the sharpness of this inequality. Furthermore, for these matrices, we find that
$(\|T_1\|+\|T_2\|)^2\approx 59.4117,$ showing how this new inequality can provide a better estimate than the triangle inequality.

Moreover, for these matrices, we have $(2\omega(T_1+T_2))^2\approx 91.2676,$ showing also that this new bound provides a better estimate than the basic bound $\|T_1+T_2\|\leq 2\omega(T_1+T_2).$ These observations help better acknowledge the value of this bound.
\end{remark}

\begin{remark}
If we put ${{T}_{1}}=T$ and ${{T}_{2}}={{T}^{*}}$ in Remark \ref{Rem_1}, we infer that
	\[\begin{aligned}
   4\left\| \Re T \right\|^2&={{\left\| T+{{T}^{*}} \right\|}^{2}} \\ 
 & \le \omega \left( {{T}^{*}}\left( T+{{T}^{*}} \right) \right)+\omega \left( T\left( T+{{T}^{*}} \right) \right) \\ 
 & =2\left( \omega \left( {{T}^{*}}\left( \Re T \right) \right)+\omega \left( T\left( \Re T \right) \right) \right). 
\end{aligned}\]
That is, if $T\in\mathbb{B}(\mathbb{H})$,
\begin{equation}\label{Eq_Ned_001}
\left\| \Re T \right\|^2\le \frac{1}{2}\left( \omega \left( {{T}^{*}}\left( \Re T \right) \right)+\omega \left( T\left( \Re T \right) \right) \right).
\end{equation}	
	
	If we let $\left[
\begin{array}{cc}
 -1 & 3 \\
 -3 & -2 \\
\end{array}
\right],$ we find that
\[\left\| \Re T \right\|^2=4, \frac{1}{2}\left( \omega \left( {{T}^{*}}\left( \Re T \right) \right)+\omega \left( T\left( \Re T \right) \right) \right)\approx 5.31843,\]
while $\omega^2(T)\approx 11.3143.$ This shows that \eqref{Eq_Ned_001} can provide a considerable improvement of the well-known bound $\|\Re T\|\leq\omega(T).$ Again, this is not always the case, showing that this new bound is an independent new bound.
\end{remark}

\begin{remark}
If we put ${{T}_{1}}=\Re T$ and ${{T}_{2}}=i\Im T$, we obtain
	\[\begin{aligned}
   {{\left\| T \right\|}^{2}}&={{\left\| \Re T+i\Im T \right\|}^{2}} \\ 
 & \le \omega \left( \Re T\left( \Re T+i\Im T \right) \right)+\omega \left( -i\Im T\left( \Re T+i\Im T \right) \right) \\ 
 & =\omega \left( \Re T\left( \Re T+i\Im T \right) \right)+\omega \left( \Im T\left( \Re T+i\Im T \right) \right) \\ 
 & =\omega \left( \Re T\left( T \right) \right)+\omega \left( \Im T\left( T \right) \right)  
\end{aligned}\]
\begin{equation}\label{Eq_Ned_01}
{{\left\| T \right\|}^{2}}\le \omega \left( \Re T\left( T \right) \right)+\omega \left( \Im T\left( T \right) \right).
\end{equation}

	We notice, first, that \eqref{Eq_Ned_01} is sharp. This can be seen by letting 
	$T=\left[
\begin{array}{cc}
 3 & 2 \\
 -2 & -3 \\
\end{array}
\right].$ Then direct calculations show that
\[\|T\|^2=\omega \left( \Re T\left( T \right) \right)+\omega \left( \Im T\left( T \right) \right)=25.\]

\end{remark}

\begin{remark}
We have shown that
	\[{{\left\| {{T}_{1}}+{{T}_{2}} \right\|}^{2}}\le \omega \left( T_{1}^{*}\left( {{T}_{1}}+{{T}_{2}} \right) \right)+\omega \left( T_{2}^{*}\left( {{T}_{1}}+{{T}_{2}} \right) \right).\]
That is,
	\[\begin{aligned}
   {{\left\| {{T}_{1}}+{{T}_{2}} \right\|}^{2}}&\le \omega \left( T_{1}^{*}{{T}_{1}}+T_{1}^{*}{{T}_{2}} \right)+\omega \left( T_{2}^{*}{{T}_{1}}+T_{2}^{*}{{T}_{2}} \right) \\ 
 & \le \left\| T_{1}^{*}{{T}_{1}}+T_{1}^{*}{{T}_{2}} \right\|+\left\| T_{2}^{*}{{T}_{1}}+T_{2}^{*}{{T}_{2}} \right\|.  
\end{aligned}\]
If we replace ${{T}_{2}}$ by ${{e}^{i\theta }}{{T}_{2}}$, we infer that
	\[{{\left\| {{T}_{1}}+{{e}^{i\theta }}{{T}_{2}} \right\|}^{2}}\le \left\| T_{1}^{*}{{T}_{1}}+{{e}^{i\theta }}T_{1}^{*}{{T}_{2}} \right\|+\left\| T_{2}^{*}{{T}_{2}}+{{e}^{-i\theta }}T_{2}^{*}{{T}_{1}} \right\|.\]
Now, by taking the supremum over $\theta \in \mathbb{R}$, \eqref{Eq_Ident_Bl} implies
\begin{equation}\label{Eq_Ned_02}	
	{{\omega }^{2}}\left( \left[ \begin{matrix}
   O & {{T}_{1}}  \\
   T_{2}^{*} & O  \\
\end{matrix} \right] \right)\le \frac{1}{2}\left( \omega \left( \left[ \begin{matrix}
   O & T_{1}^{*}{{T}_{1}}  \\
   T_{2}^{*}{{T}_{1}} & O  \\
\end{matrix} \right] \right)+\omega \left( \left[ \begin{matrix}
   O & T_{2}^{*}{{T}_{2}}  \\
   T_{1}^{*}{{T}_{2}} & O  \\
\end{matrix} \right] \right) \right).
\end{equation}
If we let
\(T_1=\left[
\begin{array}{cc}
 -2 & 0 \\
 0 & 1 \\
\end{array}
\right]\) and $T_2=\left[
\begin{array}{cc}
 -1 & 1 \\
 -2 & 2 \\
\end{array}
\right],$ we find that
\[{{\omega }^{2}}\left( \left[ \begin{matrix}
   O & {{T}_{1}}  \\
   T_{2}^{*} & O  \\
\end{matrix} \right] \right)\approx 5.15604, \frac{1}{2}\left( \omega \left( \left[ \begin{matrix}
   O & T_{1}^{*}{{T}_{1}}  \\
   T_{2}^{*}{{T}_{1}} & O  \\
\end{matrix} \right] \right)+\omega \left( \left[ \begin{matrix}
   O & T_{2}^{*}{{T}_{2}}  \\
   T_{1}^{*}{{T}_{2}} & O  \\
\end{matrix} \right] \right) \right)=2.25.\]
On the other hand, $\left(\frac{\|T_1\|+\|T_2\|}{2}\right)^2\approx 6.66228.$ This example provides an evidence that the bound found in \eqref{Eq_Ned_02} can be better than that in \eqref{Eq_Bound_Bl_1}.
\end{remark}

Another upper bound for the norm of the sum of $n$ operators can be found as follows.
\begin{theorem}\label{3}
Let ${{T}_{1}},{{T}_{2}},\ldots ,{{T}_{n}}\in \mathbb B\left( \mathbb H \right)$. Then 
\[{{\left\| \sum\limits_{k=1}^{n}{{{T}_{k}}} \right\|}^{2}}\le \left\| \sum\limits_{k=1}^{n}{T_{k}^{*}{{T}_{k}}}+\frac{1}{2}\left( \left( n-2 \right)\sum\limits_{k=1}^{n}{T_{k}^{*}{{T}_{k}}}+\sum\limits_{k=1}^{n}{T_{k}^{*}}\sum\limits_{k=1}^{n}{{{T}_{k}}} \right) \right\|.\]
\end{theorem}
\begin{proof}
We have, for any unit vector $x\in \mathbb H$,
\begin{equation}\label{1}
\begin{aligned}
   {{\left\| \sum\limits_{k=1}^{n}{{{T}_{k}}}x \right\|}^{2}}&=\Re\left( \sum\limits_{j=1}^{n}{\sum\limits_{k=1}^{n}{\left\langle {{T}_{k}}x,{{T}_{j}}x \right\rangle }} \right) \\ 
 & =\sum\limits_{j=1}^{n}{\sum\limits_{k=1}^{n}{\Re\left\langle {{T}_{k}}x,{{T}_{j}}x \right\rangle }} \\ 
 & =\sum\limits_{k=1}^{n}{{{\left\| {{T}_{k}}x \right\|}^{2}}}+\sum\limits_{1\le k\ne j\le n}{\Re\left\langle {{T}_{k}}x,{{T}_{j}}x \right\rangle }  
\end{aligned}
\end{equation}
On the other hand, we know that
	\[\Re\left\langle a,b \right\rangle \le \frac{1}{4}{{\left\| a+b \right\|}^{2}};\;a,b\in \mathbb H.\]
So, we get by \eqref{1} that
	\[\begin{aligned}
  & \sum\limits_{1\le i\ne j\le n}{\Re\left\langle {{T}_{k}}x,{{T}_{j}}x \right\rangle } \\ 
 & \le \frac{1}{4}\sum\limits_{1\le i\ne j\le n}{{{\left\| \left( {{T}_{k}}+{{T}_{j}} \right)x \right\|}^{2}}} \\ 
 & =\frac{1}{4}\sum\limits_{1\le i\ne j\le n}{\left\langle {{\left( {{T}_{k}}+{{T}_{j}} \right)}^{*}}\left( {{T}_{k}}+{{T}_{j}} \right)x,x \right\rangle } \\ 
 & =\frac{1}{4}\left\langle \sum\limits_{1\le i\ne j\le n}{{{\left( {{T}_{k}}+{{T}_{j}} \right)}^{*}}\left( {{T}_{k}}+{{T}_{j}} \right)}x,x \right\rangle .  
\end{aligned}\]
Therefore, we have shown that if $x\in\mathbb{H}$ is a unit vector, then
	\[\begin{aligned}
  & {{\left\| \sum\limits_{k=1}^{n}{{{T}_{k}}}x \right\|}^{2}} \\ 
 & \le \sum\limits_{k=1}^{n}{{{\left\| {{T}_{k}}x \right\|}^{2}}}+\frac{1}{4}\left\langle \sum\limits_{1\le k\ne j\le n}{{{\left( {{T}_{k}}+{{T}_{j}} \right)}^{*}}\left( {{T}_{k}}+{{T}_{j}} \right)}x,x \right\rangle  \\ 
 & =\sum\limits_{k=1}^{n}{\left\langle T_{k}^{*}{{T}_{k}}x,x \right\rangle }+\frac{1}{4}\left\langle \sum\limits_{1\le k\ne j\le n}{{{\left( {{T}_{k}}+{{T}_{j}} \right)}^{*}}\left( {{T}_{k}}+{{T}_{j}} \right)}x,x \right\rangle  \\ 
 & =\left\langle \sum\limits_{k=1}^{n}{T_{k}^{*}{{T}_{k}}}x,x \right\rangle +\frac{1}{4}\left\langle \sum\limits_{1\le k\ne j\le n}{{{\left( {{T}_{k}}+{{T}_{j}} \right)}^{*}}\left( {{T}_{k}}+{{T}_{j}} \right)}x,x \right\rangle  \\ 
 & =\left\langle \left( \sum\limits_{k=1}^{n}{T_{k}^{*}{{T}_{k}}}+\frac{1}{4}\sum\limits_{1\le k\ne j\le n}{{{\left( {{T}_{k}}+{{T}_{j}} \right)}^{*}}\left( {{T}_{k}}+{{T}_{j}} \right)} \right)x,x \right\rangle  .
\end{aligned}\]
That is,
\begin{equation}\label{6}
{{\left\| \sum\limits_{k=1}^{n}{{{T}_{k}}}x \right\|}^{2}}\le \left\langle \left( \sum\limits_{k=1}^{n}{T_{k}^{*}{{T}_{k}}}+\frac{1}{4}\sum\limits_{1\le i\ne j\le n}{{{\left( {{T}_{k}}+{{T}_{j}} \right)}^{*}}\left( {{T}_{k}}+{{T}_{j}} \right)} \right)x,x \right\rangle .
\end{equation}
Taking the supremum over all unit vectors $x\in \mathbb H$, we get
\begin{equation}\label{2}
{{\left\| \sum\limits_{k=1}^{n}{{{T}_{k}}} \right\|}^{2}}\le \left\| \sum\limits_{k=1}^{n}{T_{k}^{*}{{T}_{k}}}+\frac{1}{4}\sum\limits_{1\le i\ne j\le n}{{{\left( {{T}_{k}}+{{T}_{j}} \right)}^{*}}\left( {{T}_{k}}+{{T}_{j}} \right)} \right\|.
\end{equation}
On the other hand, we know that
\begin{equation}\label{7}
\begin{aligned}
  & \sum\limits_{1\le k\ne j\le n}{{{\left( {{T}_{k}}+{{T}_{j}} \right)}^{*}}\left( {{T}_{k}}+{{T}_{j}} \right)} \\ 
 & =\sum\limits_{k,j=1}^{n}{\left( T_{k}^{*}{{T}_{k}}+T_{j}^{*}{{T}_{j}}+T_{k}^{*}{{T}_{j}}+T_{j}^{*}{{T}_{k}} \right)}-4\sum\limits_{k=1}^{n}{T_{k}^{*}{{T}_{k}}} \\ 
 & =2n\sum\limits_{k=1}^{n}{T_{k}^{*}{{T}_{k}}}+2\sum\limits_{k=1}^{n}{T_{k}^{*}}\sum\limits_{k=1}^{n}{{{T}_{k}}}-4\sum\limits_{k=1}^{n}{T_{k}^{*}{{T}_{k}}} \\ 
 & =2\left( \left( n-2 \right)\sum\limits_{k=1}^{n}{T_{k}^{*}{{T}_{k}}}+\sum\limits_{k=1}^{n}{T_{k}^{*}}\sum\limits_{k=1}^{n}{{{T}_{k}}} \right),
\end{aligned}
\end{equation}
so, we infer from \eqref{2} that
	\[{{\left\| \sum\limits_{k=1}^{n}{{{T}_{k}}} \right\|}^{2}}\le \left\| \sum\limits_{k=1}^{n}{T_{k}^{*}{{T}_{k}}}+\frac{1}{2}\left( \left( n-2 \right)\sum\limits_{k=1}^{n}{T_{k}^{*}{{T}_{k}}}+\sum\limits_{k=1}^{n}{T_{k}^{*}}\sum\limits_{k=1}^{n}{{{T}_{k}}} \right) \right\|,\]
as required.
\end{proof}

\begin{corollary}\label{4}
Let ${{T}_{1}},{{T}_{2}}\in \mathbb B\left( \mathbb H \right)$. Then 
	\[{{\left\| {{T}_{1}}+{{T}_{2}} \right\|}^{2}}\le \left\| \frac{3}{2}\left( {{\left| {{T}_{1}} \right|}^{2}}+{{\left| {{T}_{2}} \right|}^{2}} \right)+\Re\left( T_{1}^{*}{{T}_{2}} \right) \right\|.\]
\end{corollary}
\begin{proof}
The case $n=2$, in Theorem \ref{3}, gives
\[{{\left\| {{T}_{1}}+{{T}_{2}} \right\|}^{2}}\le \left\| T_{1}^{*}{{T}_{1}}+T_{2}^{*}{{T}_{2}}+\frac{1}{2}{{\left| {{T}_{1}}+{{T}_{2}} \right|}^{2}} \right\|.\]
Since 
	\[\begin{aligned}
   {{\left\| {{T}_{1}}+{{T}_{2}} \right\|}^{2}}&\le \left\| T_{1}^{*}{{T}_{1}}+T_{2}^{*}{{T}_{2}}+\frac{1}{2}\left( T_{1}^{*}+T_{2}^{*} \right)\left( {{T}_{1}}+{{T}_{2}} \right) \right\| \\ 
 & =\left\| T_{1}^{*}{{T}_{1}}+T_{2}^{*}{{T}_{2}}+\frac{1}{2}{{\left( {{T}_{1}}+{{T}_{2}} \right)}^{*}}\left( {{T}_{1}}+{{T}_{2}} \right) \right\| \\ 
 & =\left\| T_{1}^{*}{{T}_{1}}+T_{2}^{*}{{T}_{2}}+\frac{1}{2}{{\left| {{T}_{1}}+{{T}_{2}} \right|}^{2}} \right\| \\ 
 & =\left\| \frac{3}{2}\left( {{\left| {{T}_{1}} \right|}^{2}}+{{\left| {{T}_{2}} \right|}^{2}} \right)+\Re\left( T_{1}^{*}{{T}_{2}} \right) \right\|,  
\end{aligned}\]
we have
	\[{{\left\| {{T}_{1}}+{{T}_{2}} \right\|}^{2}}\le \left\| \frac{3}{2}\left( {{\left| {{T}_{1}} \right|}^{2}}+{{\left| {{T}_{2}} \right|}^{2}} \right)+\Re\left( T_{1}^{*}{{T}_{2}} \right) \right\|,\]
as required.
\end{proof}

\begin{remark}
In this remark, we show that the inequality 
\[{{\left\| {{T}_{1}}+{{T}_{2}} \right\|}^{2}}\le \left\| \frac{3}{2}\left( {{\left| {{T}_{1}} \right|}^{2}}+{{\left| {{T}_{2}} \right|}^{2}} \right)+\Re\left( T_{1}^{*}{{T}_{2}} \right) \right\|\]
given  in Corollary \ref{4} is sharp. Indeed, let $T_1=\left[
\begin{array}{cc}
 -2 & 1 \\
 0 & -2 \\
\end{array}
\right]$ and $T_2=\left[
\begin{array}{cc}
 -2 & -1 \\
 0 & 2 \\
\end{array}
\right].$ Then it can be seen that
\[\|T_1+T_2\|^2= \left\| \frac{3}{2}\left( {{\left| {{T}_{1}} \right|}^{2}}+{{\left| {{T}_{2}} \right|}^{2}} \right)+\Re\left( T_{1}^{*}{{T}_{2}} \right) \right\|=16.\]

The other observation is that the corollary can provide better estimates than the triangle inequality in many cases. This can be seen, for example, by considering the above matrices, where we can see that
\[(\|T_1\|+\|T_2\|)^2=2 (9 + \sqrt{17})\approx 26.2462.\]
We point out that this is not always the case, as the triangle inequality can provide better estimates in other situations, revealing that the two bounds are, in general, incomparable.   
\end{remark}

\begin{corollary}\label{Cor_Block_two}
Let $T,{{T}_{1}},{{T}_{2}}\in \mathbb B\left( \mathbb H \right)$. Then 
	\[{{\omega }^{2}}\left( \left[ \begin{matrix}
   O & {{T}_{1}}  \\
   T_{2}^{*} & O  \\
\end{matrix} \right] \right)\le \frac{1}{2}\omega \left( \begin{matrix}
   O & \frac{3}{2}\left( {{\left| {{T}_{1}} \right|}^{2}}+{{\left| {{T}_{2}} \right|}^{2}} \right)  \\
   T_{2}^{*}{{T}_{1}} & O  \\
\end{matrix} \right).\]
In particular,
\[{{\omega }^{2}}\left( {{T}} \right)\le \frac{1}{2}\omega \left( \begin{matrix}
   O & \frac{3}{2}\left( {{\left| {{T}} \right|}^{2}}+{{\left| T^{*} \right|}^{2}} \right)  \\
   T^{2} & O  \\
\end{matrix} \right).\]
\end{corollary}
\begin{proof}
It follows from Corollary \ref{4} that
	\[\begin{aligned}
   {{\left\| {{T}_{1}}+{{T}_{2}} \right\|}^{2}}&\le \left\| \frac{3}{2}\left( {{\left| {{T}_{1}} \right|}^{2}}+{{\left| {{T}_{2}} \right|}^{2}} \right)+\Re\left( T_{1}^{*}{{T}_{2}} \right) \right\| \\ 
 & =\left\| \Re\left( \frac{3}{2}\left( {{\left| {{T}_{1}} \right|}^{2}}+{{\left| {{T}_{2}} \right|}^{2}} \right)+T_{1}^{*}{{T}_{2}} \right) \right\| \\ 
 & \le \omega \left( \frac{3}{2}\left( {{\left| {{T}_{1}} \right|}^{2}}+{{\left| {{T}_{2}} \right|}^{2}} \right)+T_{1}^{*}{{T}_{2}} \right).  
\end{aligned}\]
Indeed, we have shown that
\begin{equation}\label{5}
{{\left\| {{T}_{1}}+{{T}_{2}} \right\|}^{2}}\le \omega \left( \frac{3}{2}\left( {{\left| {{T}_{1}} \right|}^{2}}+{{\left| {{T}_{2}} \right|}^{2}} \right)+T_{1}^{*}{{T}_{2}} \right).
\end{equation}
If we replace ${{T}_{2}}$ by ${{e}^{i\theta }}{{T}_{2}}$, in \eqref{5}, we get
	\[\begin{aligned}
   \frac{1}{4}{{\left\| {{T}_{1}}+{{e}^{i\theta }}{{T}_{2}} \right\|}^{2}}&\le \frac{1}{4}\omega \left( \frac{3}{2}\left( {{\left| {{T}_{1}} \right|}^{2}}+{{\left| {{T}_{2}} \right|}^{2}} \right)+{{e}^{i\theta }}T_{1}^{*}{{T}_{2}} \right) \\ 
 & \le \frac{1}{4}\left\| \frac{3}{2}\left( {{\left| {{T}_{1}} \right|}^{2}}+{{\left| {{T}_{2}} \right|}^{2}} \right)+{{e}^{i\theta }}T_{1}^{*}{{T}_{2}} \right\| \\ 
 & \le \frac{1}{2}\omega \left( \begin{matrix}
   O & \frac{3}{2}\left( {{\left| {{T}_{1}} \right|}^{2}}+{{\left| {{T}_{2}} \right|}^{2}} \right)  \\
   T_{2}^{*}{{T}_{1}} & O  \\
\end{matrix} \right).  
\end{aligned}\]
That is,
	\[\frac{1}{4}{{\left\| {{T}_{1}}+{{e}^{i\theta }}{{T}_{2}} \right\|}^{2}}\le \frac{1}{2}\omega \left( \begin{matrix}
   O & \frac{3}{2}\left( {{\left| {{T}_{1}} \right|}^{2}}+{{\left| {{T}_{2}} \right|}^{2}} \right)  \\
   T_{2}^{*}{{T}_{1}} & O  \\
\end{matrix} \right).\]
Now, by taking the supremum over $\theta \in \mathbb{R}$, \eqref{Eq_Ident_Bl} implies
	\[{{\omega }^{2}}\left( \left[ \begin{matrix}
   O & {{T}_{1}}  \\
   T_{2}^{*} & O  \\
\end{matrix} \right] \right)\le \frac{1}{2}\omega \left( \begin{matrix}
   O & \frac{3}{2}\left( {{\left| {{T}_{1}} \right|}^{2}}+{{\left| {{T}_{2}} \right|}^{2}} \right)  \\
   T_{2}^{*}{{T}_{1}} & O  \\
\end{matrix} \right).\]
In particular, if we replace $T_{2}^{*}$ by ${{T}_{1}}$, we get
	\[\begin{aligned}
   {{\omega }^{2}}\left( {{T}_{1}} \right)&={{\omega }^{2}}\left( \left[ \begin{matrix}
   O & {{T}_{1}}  \\
   {{T}_{1}} & O  \\
\end{matrix} \right] \right) \\ 
 & \le \frac{1}{2}\omega \left( \begin{matrix}
   O & \frac{3}{2}\left( {{\left| {{T}_{1}} \right|}^{2}}+{{\left| T_{1}^{*} \right|}^{2}} \right)  \\
   T_{1}^{2} & O  \\
\end{matrix} \right).  
\end{aligned}\]
Namely,
	\[{{\omega }^{2}}\left( {{T}_{1}} \right)\le \frac{1}{2}\omega \left( \begin{matrix}
   O & \frac{3}{2}\left( {{\left| {{T}_{1}} \right|}^{2}}+{{\left| T_{1}^{*} \right|}^{2}} \right)  \\
   T_{1}^{2} & O  \\
\end{matrix} \right).\]
We deduce the desired result by replacing $T_1$ by $T$.
\end{proof}	
\begin{remark}
So, Corollary \ref{Cor_Block_two} provides an upper bound for ${{\omega }^{2}}\left( \left[ \begin{matrix}
   O & {{T}_{1}}  \\
   T_{2}^{*} & O  \\
\end{matrix} \right] \right)$. We give an example here that shows how this new bound can be better than that in \eqref{Eq_Bound_Bl_1}. Indeed, if we let $T_1=\left[
\begin{array}{cc}
 -3 & 3 \\
 1 & 0 \\
\end{array}
\right]$ and $T_2=\left[
\begin{array}{cc}
 -1 & -3 \\
 1 & 3 \\
\end{array}
\right],$ then we can see that
\[{{\omega }^{2}}\left( \left[ \begin{matrix}
   O & {{T}_{1}}  \\
   T_{2}^{*} & O  \\
\end{matrix} \right] \right)\approx 12.0635, \frac{1}{2}\omega \left( \begin{matrix}
   O & \frac{3}{2}\left( {{\left| {{T}_{1}} \right|}^{2}}+{{\left| {{T}_{2}} \right|}^{2}} \right)  \\
   T_{2}^{*}{{T}_{1}} & O  \\
\end{matrix} \right)\approx 13.1313,\]
while $\left(\frac{\|T_1\|+\|T_2\|}{2}\right)^2\approx 19.2498.$ However, this is not always the case, as there are other examples with the opposite conclusion.
\end{remark}

We conclude with the following singular value inequality.
\begin{theorem}
Let ${{T}_{1}},{{T}_{2}},\ldots ,{{T}_{n}}\in \mathbb K\left( \mathbb H \right)$. Then for $j=1,2,\ldots$
\[s_{j}^{2}\left( \sum\limits_{k=1}^{n}{{{T}_{k}}} \right)\le {{s}_{j}}\left( \sum\limits_{k=1}^{n}{T_{k}^{*}{{T}_{k}}}+\frac{1}{2}\left( \left( n-2 \right)\sum\limits_{k=1}^{n}{T_{k}^{*}{{T}_{k}}}+\sum\limits_{k=1}^{n}{T_{k}^{*}}\sum\limits_{k=1}^{n}{{{T}_{k}}} \right) \right).\]
\end{theorem}
\begin{proof}
In \eqref{6}, we have shown that if $x\in\mathbb{H}$ is a unit vector, then
{\small
	\[\begin{aligned}
  & {{\left\| \sum\limits_{k=1}^{n}{{{T}_{k}}}x \right\|}^{2}} \\ 
 & \le \left\langle \left( \sum\limits_{k=1}^{n}{T_{k}^{*}{{T}_{k}}}+\frac{1}{4}\sum\limits_{1\le k\ne j\le n}{{{\left( {{T}_{k}}+{{T}_{j}} \right)}^{*}}\left( {{T}_{k}}+{{T}_{j}} \right)} \right)x,x \right\rangle  \\ 
 & =\left\langle {{\left| \sum\limits_{k=1}^{n}{T_{k}^{*}{{T}_{k}}}+\frac{1}{4}\sum\limits_{1\le k\ne j\le n}{{{\left( {{T}_{k}}+{{T}_{j}} \right)}^{*}}\left( {{T}_{k}}+{{T}_{j}} \right)} \right|}^{\frac{1}{2}}}x,{{\left| \sum\limits_{k=1}^{n}{T_{k}^{*}{{T}_{k}}}+\frac{1}{4}\sum\limits_{1\le k\ne j\le n}{{{\left( {{T}_{k}}+{{T}_{j}} \right)}^{*}}\left( {{T}_{k}}+{{T}_{j}} \right)} \right|}^{\frac{1}{2}}}x \right\rangle  \\ 
 & ={{\left\| {{\left| \sum\limits_{k=1}^{n}{T_{k}^{*}{{T}_{k}}}+\frac{1}{4}\sum\limits_{1\le k\ne j\le n}{{{\left( {{T}_{k}}+{{T}_{j}} \right)}^{*}}\left( {{T}_{k}}+{{T}_{j}} \right)} \right|}^{\frac{1}{2}}}x \right\|}^{2}}.  
\end{aligned}\]
}
So,
	\[\left\| \sum\limits_{k=1}^{n}{{{T}_{k}}}x \right\|\le \left\| {{\left| \sum\limits_{k=1}^{n}{T_{k}^{*}{{T}_{k}}}+\frac{1}{4}\sum\limits_{1\le k\ne j\le n}{{{\left( {{T}_{k}}+{{T}_{j}} \right)}^{*}}\left( {{T}_{k}}+{{T}_{j}} \right)} \right|}^{\frac{1}{2}}}x \right\|.\]
Hence, by Lemma \ref{lem_maxmin},
	\[\begin{aligned}
  & {{s}_{j}}\left( \sum\limits_{k=1}^{n}{{{T}_{k}}} \right) \\ 
 & =\underset{\dim M=j}{\mathop{\max }}\,\underset{\left\| x \right\|=1}{\mathop{\underset{x\in M}{\mathop{\min }}\,}}\,\left\| \sum\limits_{k=1}^{n}{{{T}_{k}}}x \right\| \\ 
 & \le \underset{\dim M=j}{\mathop{\max }}\,\underset{\left\| x \right\|=1}{\mathop{\underset{x\in M}{\mathop{\min }}\,}}\,\left\| {{\left| \sum\limits_{k=1}^{n}{T_{k}^{*}{{T}_{k}}}+\frac{1}{4}\sum\limits_{1\le k\ne j\le n}{{{\left( {{T}_{k}}+{{T}_{j}} \right)}^{*}}\left( {{T}_{k}}+{{T}_{j}} \right)} \right|}^{\frac{1}{2}}}x \right\| \\ 
 & ={{s}_{j}}\left( {{\left| \sum\limits_{k=1}^{n}{T_{k}^{*}{{T}_{k}}}+\frac{1}{4}\sum\limits_{1\le k\ne j\le n}{{{\left( {{T}_{k}}+{{T}_{j}} \right)}^{*}}\left( {{T}_{k}}+{{T}_{j}} \right)} \right|}^{\frac{1}{2}}} \right).
\end{aligned}\]
Therefore,
	\[\begin{aligned}
   {{s}_{j}}\left( \sum\limits_{k=1}^{n}{{{T}_{k}}} \right)&\le {{s}_{j}}\left( {{\left| \sum\limits_{k=1}^{n}{T_{k}^{*}{{T}_{k}}}+\frac{1}{4}\sum\limits_{1\le k\ne j\le n}{{{\left( {{T}_{k}}+{{T}_{j}} \right)}^{*}}\left( {{T}_{k}}+{{T}_{j}} \right)} \right|}^{\frac{1}{2}}} \right) \\ 
 & =s_{j}^{\frac{1}{2}}\left( \left| \sum\limits_{k=1}^{n}{T_{k}^{*}{{T}_{k}}}+\frac{1}{4}\sum\limits_{1\le k\ne j\le n}{{{\left( {{T}_{k}}+{{T}_{j}} \right)}^{*}}\left( {{T}_{k}}+{{T}_{j}} \right)} \right| \right) \\ 
 & =s_{j}^{\frac{1}{2}}\left( \sum\limits_{k=1}^{n}{T_{k}^{*}{{T}_{k}}}+\frac{1}{4}\sum\limits_{1\le k\ne j\le n}{{{\left( {{T}_{k}}+{{T}_{j}} \right)}^{*}}\left( {{T}_{k}}+{{T}_{j}} \right)} \right).  
\end{aligned}\]
That is,
	\[s_{j}^{2}\left( \sum\limits_{k=1}^{n}{{{T}_{k}}} \right)\le {{s}_{j}}\left( \sum\limits_{k=1}^{n}{T_{k}^{*}{{T}_{k}}}+\frac{1}{4}\sum\limits_{1\le k\ne j\le n}{{{\left( {{T}_{k}}+{{T}_{j}} \right)}^{*}}\left( {{T}_{k}}+{{T}_{j}} \right)} \right).\]
By \eqref{7}, we infer that
\[s_{j}^{2}\left( \sum\limits_{k=1}^{n}{{{T}_{k}}} \right)\le {{s}_{j}}\left( \sum\limits_{k=1}^{n}{T_{k}^{*}{{T}_{k}}}+\frac{1}{2}\left( \left( n-2 \right)\sum\limits_{k=1}^{n}{T_{k}^{*}{{T}_{k}}}+\sum\limits_{k=1}^{n}{T_{k}^{*}}\sum\limits_{k=1}^{n}{{{T}_{k}}} \right) \right),\]
as required.
\end{proof}

\vspace{.25in}

{\tiny (Z. I. Ismailov) Department of Mathematics, Karadeniz Technical University, Trabzon, Turkey}

{\tiny \textit{E-mail address:} zameddin.ismailov@gmail.com}

\vskip 0.3 true cm 	

{\tiny (P. I. Al) Department of Mathematics, Karadeniz Technical University, Trabzon, Turkey}
	
{\tiny\textit{E-mail address:} ipekpembe@gmail.com}

\vskip 0.3 true cm 	

{\tiny (H. R. Moradi) Department of Mathematics, Ma.C., Islamic Azad University, Mashhad, Iran 
	
\textit{E-mail address:} hrmoradi@mshdiau.ac.ir}

\vskip 0.3 true cm 	

{\tiny (M. Sababheh) Department of Basic Sciences, Princess Sumaya University for Technology, Amman, Jordan
	
\textit{E-mail address:} sababheh@yahoo.com}
%########################

\end{document}